\documentclass[a4paper]{amsart} 
\usepackage[utf8]{inputenc}
\usepackage{microtype}
\usepackage{amsmath,amssymb,amsthm}
\usepackage{mathtools}
\usepackage{enumitem}
\usepackage{tikz-cd}

\usepackage{hyperref}

\newcommand{\Z}{\mathbb Z}
\newcommand{\Q}{\mathbb Q}
\newcommand{\F}{\mathbb F}
\newcommand{\ol}[1]{\overline{#1}}
\newcommand{\set}[2]{\left\{#1\,\middle|\,#2\right\}}
\newcommand{\sm}{\mathrm{sm}}
\newcommand{\smooth}{(\,\cdot\,)^\sm} 
\newcommand{\abs}{\mathrm{abs}}
\newcommand{\cont}{\mathrm{cont}}
\newcommand{\RR}{\mathbf{R}}
\newcommand{\ie}{\textit{i.e.}, }

\DeclareMathOperator{\Ext}{Ext}
\DeclareMathOperator{\GL}{GL}
\DeclareMathOperator{\ind}{ind}
\DeclareMathOperator{\Ker}{Ker}
\DeclareMathOperator{\Rep}{Rep}

\newtheorem{thm}{Theorem}

\newtheorem{cor}[thm]{Corollary}
\newtheorem{lem}[thm]{Lemma}

\theoremstyle{definition}
\newtheorem*{defn*}{Definition}

\theoremstyle{remark}
\newtheorem{ex}[thm]{Example}
\newtheorem{rmk*}{Remark}

\title{On the Smooth Part Functor}
\author{Claudius Heyer}
\address{Mathematisches Institut, Westf\"alische Wilhelms-Universit\"at
M\"unster, Einsteinstra\ss{}e 62, D-48149 M\"unster, Germany}
\email{cheyer@uni-muenster.de}

\begin{document}
\begin{abstract} 
Let $G$ be a compact $p$-adic analytic group and $k$ a field positive
characteristic. We prove that for every smooth representation of $G$ on
a $k$-vector space $V$, every 1-cocycle $G\to V$ is continuous. We deduce that
the first derived functor of the smooth part functor vanishes on smooth
representations. As a corollary, we obtain that extensions of smooth
representations are automatically smooth.
\end{abstract} 
\maketitle

\section{Main Result} 
Let $k$ be a commutative ring of characteristic $m>0$. Given a locally profinite
group $G$, we denote by $\Rep_k(G)$ the category of representations of
$G$ on $k$-modules. Let $\Rep_k^{\sm}(G)$ be the full subcategory of $\Rep_k(G)$
consisting of the \emph{smooth} representations, that is, those representations
$V$ of $G$ for which the action map $G\times V\to V$ is continuous if we endow
$V$ with the discrete topology. 

Assume now that $G$ is a compact $p$-adic analytic group. We can view $G$ either
as an abstract group or as a topological group with the profinite
topology. Accordingly, there are two cohomology theories associated with $G$
taking values
in $\Rep_k^\sm(G)$: We denote by $H^q_{\abs}(G,V)$ the $q$-th cohomology group
of $G$  when we view $G$ as an abstract group. When $G$ is viewed as a
topological group we denote by $H^q_{\cont}(G,V)$ the $q$-th continuous (or
profinite, or Galois) cohomology group.
There is an obvious comparison homomorphism
\[
\varphi^q\colon H^q_{\cont}(G,V) \longrightarrow H^q_{\abs}(G,V),\qquad
\text{for $q\ge0$,}
\]
and it is a general question when this is an isomorphism. When $G$ is a pro-$p$
group and $V=\F_p$, this question was raised and studied in
\cite{Comparison.2007}. 

The reason why we restrict to $p$-adic analytic groups and coefficient rings of
positive characteristic is the following example communicated to me by
E.~Bodon: 
\begin{ex}\label{ex} 
Let $k$ be a non-archimedean local field, that is, either a finite
extension of $\Q_p$ or $\F_{p^f}((t))$, for some $f\ge1$. Let $G$ be the
underlying abelian group of $k$. Note that either $k$ does not have positive
characteristic or $G$ is not (topologically) finitely generated. Consider the
two-dimensional $G$-representation $V = k^2$ given by
\[
G\longrightarrow \GL_2(k),\quad x\longmapsto
\begin{pmatrix}1&x\\0&1\end{pmatrix}.
\]
We obtain a short exact sequence $0\to k\to V\to k\to 0$, where the outer terms
are smooth but $V$ is not. Equivalently, the corresponding $1$-cocycle
$\chi\colon G\to k$, $\chi(x) = x$ is not continuous (because $k$ carries the
discrete topology whereas $G$ does not). Therefore, the map $\varphi^1$ above
is not surjective.
\end{ex} 

It is easy to see that $\varphi^1$ is always injective. Our main result is that
$\varphi^1$ is also surjective. More precisely, we prove:

\begin{thm}\label{thm:main} 
Let $V \in \Rep_k^\sm(G)$. Every $1$-cocycle $G\to V$ is automatically
continuous. In particular, the comparison map $\varphi^1$ is an isomorphism.
\end{thm} 

If $G$ is a finitely generated pro-$p$ group acting trivially on $V$,
Theorem~\ref{thm:main} follows from the general result, due to Serre
\cite[\S4.2, Exercise~6]{Serre.2013}, that every group homomorphism $\chi\colon
G\to V$ is continuous. 

\section{Recollection of Powerful pro-\texorpdfstring{$p$}{p} Groups} 
It follows from \cite[8.1~Theorem]{DDMS99} that a topological group $G$ is
$p$-adic analytic if and only if $G$ contains an open subgroup which is a
finitely generated powerful pro-$p$ group. We collect here some properties of
powerful pro-$p$ groups following \cite{DDMS99}.

Recall that a pro-$p$ group
$G$ is called \emph{powerful} if $p$ is odd and $G/\ol{G^p}$ is abelian, or
$p=2$ and $G/\ol{G^4}$ is abelian. Here, the bar
denotes topological closure and $G^n$, $n\in\Z_{>0}$, denotes the subgroup of
$G$ generated by the subset $\set{g^n}{g\in G}$.
The group $G$ is called \emph{uniformly powerful} if it is (topologically)
finitely generated, powerful, and torsionfree.

Fix a finitely generated, powerful pro-$p$ group $G$. We collect some properties:
\begin{enumerate}[label=(\Alph*)] 
\item\label{powerful-A} The subgroup $G^p = \ol{[G,G]G^p}$ is open. If $p=2$,
then $G^4$ is open \cite[Lem.~3.4]{DDMS99}.

\item\label{powerful-B} Writing $G = \ol{\langle a_1,\dotsc,a_d\rangle}$, we
have $G = \ol{\langle a_1\rangle}\dotsm \ol{\langle a_d\rangle}$, \ie $G$
is a product of procyclic subgroups \cite[Prop.~3.7]{DDMS99}.

\item\label{powerful-C} Every closed subgroup $H$ of $G$ is again a finitely
generated pro-$p$ group, and the cardinality of a minimal generating
set for $H$ is bounded above by the one for $G$ \cite[Thm.~3.8]{DDMS99}.

\item\label{powerful-D} Let $H$ be a finitely generated pro-$p$ group, and
suppose there exists $r\in \Z_{>0}$ such that every open normal subgroup of $H$
can be generated by at most $r$ elements. Then $H$ has a powerful,
characteristic, open subgroup \cite[Thm.~3.10]{DDMS99}.

\item\label{powerful-E} $G$ contains a characteristic, open, uniformly powerful
subgroup \cite[Cor.~4.3 and Thm.~3.13]{DDMS99}.

\item\label{powerful-F} If $G$ is uniformly powerful and $\{a_1,\dotsc,a_d\}$ is
a minimal generating set, then the mapping
\[
\Z_p^d \longrightarrow G,\quad (\lambda_1,\dotsc,\lambda_d) \longmapsto
a_1^{\lambda_1}\dotsm a_d^{\lambda_d}
\]
is a well-defined homeomorphism (but \emph{not} a group homomorphism)
\cite[Thm.~4.9]{DDMS99}.
\end{enumerate} 

It follows from properties \ref{powerful-C}, \ref{powerful-D}, and
$\ref{powerful-E}$ above that $G$ admits a fundamental basis $\Omega(G)$ of open
neighborhoods of $1$ consisting of uniformly powerful, normal subgroups.

\section{Proof of the Theorem} 
Recall that we assumed that $G$ is a compact $p$-adic analytic group. We first
prove a special case of Theorem~\ref{thm:main}

\begin{lem}\label{lem:abelian} 
Assume that $G$ is abelian. Let $V\in \Rep_k^\sm(G)$. Every $1$-cocycle $G\to
V$ is automatically continuous.
\end{lem} 
\begin{proof} 
Let $V\in \Rep_k^\sm(G)$ and let $\chi\colon G\to V$ be a $1$-cocycle. We argue
that it suffices to find an open pro-$p$ subgroup $U\le G$ with $\chi(G)
\subseteq V^U$. Given such $U$ the map $\chi\big|_U\colon U\to V^U$ is a group
homomorphism. Note that $U$ is powerful, since it is abelian, and finitely
generated, since $G$ is $p$-adic analytic. Using \ref{powerful-A}, one easily
verifies that the subgroup $U^{p^2m}$ is open in $U$. From $m.V = \{0\}$ it
follows that $\chi$ vanishes on $U^{p^2m}$. The cocycle condition then
ensures that $\chi$ is locally constant and hence continuous.

For each $g\in G$ there exists a compact open subgroup $K_g\le G$ with
$\chi(g)\in V^{K_g}$.
As $G$ is finitely generated, there exist $a_1,\dotsc,a_d$ in $G$ generating a
dense subgroup; denote it by $A$. Let $U\in \Omega(G)$ be a subgroup contained
in $\bigcap_{i=1}^dK_{a_i}$. Then $\chi(a_i) \in V^U$ for all
$i$. Given $g,h\in G$ with $\chi(g), \chi(h)\in V^U$, we have
\[
\chi(gh^{-1}) = \chi(g) - (gh^{-1})\chi(h) \in V^U,
\]
because $G$ normalizes $U$ and hence acts on $V^U$. In particular, $\chi(A)
\subseteq V^U$. As noted at the beginning of the proof, we are done once we show
$\chi(G) \subseteq V^U$.

Fix any $u\in U$. For each $g\in K_u$ we compute, using that $G$ is abelian,
\[
u.\chi(g) = \chi(ug) - \chi(u) = \chi(gu) - \chi(u) = \chi(g) + g.\chi(u) -
\chi(u) = \chi(g),
\]
\ie $\chi(K_u) \subseteq V^u$. As $K_u.A = G$ and $\chi(A)\subseteq
V^U\subseteq V^u$, we conclude $\chi(G)\subseteq V^u$. As $u$ was arbitrary, we
deduce $\chi(G)\subseteq \bigcap_{u\in U}V^u = V^U$. This finishes the proof.
\end{proof} 

\begin{proof}[Proof of Theorem~\ref{thm:main}] 
Let $\chi\colon G\to V$ be a cocycle. Write $G = \ol{\langle
a_1,\dotsc,a_d\rangle}$ such that $d$ is minimal. Then property~\ref{powerful-F}
implies that the multiplication map induces a homeomorphism
\begin{equation}\label{eq:homeo}
H_1\times\dotsb \times H_d\xrightarrow{\cong} G,
\end{equation}
where $H_i = \ol{\langle a_i\rangle} \cong \Z_p$, $i=1,\dotsc,d$. Denote by
$W_i$ the $H_i$-representation obtained from $V$ via restriction.
Lemma~\ref{lem:abelian} shows that $\chi\big|_{H_i}\colon H_i\to W_i$ is
continuous. In particular, there exists an open subgroup $U_i$ of $H_i$ with
$\chi(U_i) = \{0\}$. It follows from \eqref{eq:homeo} that the subset $U_1\dotsm
U_d$ of $G$ is an open neighborhood of $1$, hence it contains an open normal
subgroup $U$ of $G$. It follows from the cocycle property that $\chi(U) =
\{0\}$ and further that $\chi\colon G\to V$ is locally constant, that is,
continuous. 
\end{proof} 
\section{Application to the Smooth Part Functor} 
Let $k$ be a field of positive characteristic. Fix a $p$-adic analytic group
$G$.\footnote{Here, $p$ need not be the characteristic of $k$.} For example,
$G$ could be the group of $F$-points of a connected reductive $F$-group, where
$F$ is a finite field extension of $\Q_p$.

Given any $V\in \Rep_k(G)$, the $k$-vector space
\[
V^\sm\coloneqq \bigcup_{K\le G} V^K \subseteq V,
\]
where $V^K\coloneqq \set{v\in V}{\text{$xv=v$ for all $x\in K$}}$, $K$ runs
through the compact open subgroups of $G$, is the largest smooth
subrepresentation of $G$ contained in $V$. This yields a functor
\[
\smooth\colon \Rep_k(G)\longrightarrow \Rep_k^\sm(G).
\]
Given $V\in \Rep_k(G)$, we can also write $V^\sm = \varinjlim_{K\le G} V^K$,
where $K$ runs through the compact open subgroups of $G$. As filtered colimits
are exact, the underlying $k$-vector spaces of the right derived functors of
$\smooth$ can be computed as
\[
\RR^i\smooth(V) = \varinjlim_{K\le G} H_{\abs}^i(K,V) = \varinjlim_{K\le G}
\Ext^i_{\Rep_k(G)}(\ind_K^G(1), V),\quad\text{for all $i\ge0$.}
\]
About $\RR^1\smooth$ we have the following result:

\begin{thm}\label{thm:RR1} 
Let $V\in \Rep_k^\sm(G)$. Then $\RR^1\smooth(V) = \{0\}$.
\end{thm} 
\begin{proof} 
Applying Theorem~\ref{thm:main}, we deduce
\[
\RR^1\smooth(V) = \varinjlim_{K\le G}H_{\abs}^1(K,V) = \varinjlim_{K\le G}
H_{\cont}^1(K,V) = \{0\}.\qedhere
\]
\end{proof} 

\begin{cor}\label{cor:smoothkernel} 
Let $f\colon V\to W$ be a surjective morphism in $\Rep_k(G)$ such that $\Ker(f)$
is smooth. Then $f^\sm\colon V^\sm \to W^\sm$ is surjective.
\end{cor} 
\begin{proof} 
The short exact sequence $0\to \Ker(f)\to V\to W\to 0$ yields an exact
sequence $V^\sm \xrightarrow{f^\sm} W^\sm \to \RR^1\smooth
\bigl(\Ker(f)\bigr) = \{0\}$. Hence, $f^\sm$ is surjective.
\end{proof} 

\begin{cor}\label{cor:torsiontheory} 
Let $V\in \Rep_k(G)$. Then $(V/V^\sm)^\sm = \{0\}$.
\end{cor} 
\begin{proof} 
The quotient map $V\to V/V^\sm$ induces the zero map $V^\sm \to
(V/V^\sm)^\sm$ which, by Corollary~\ref{cor:smoothkernel}, is surjective.
\end{proof} 

\begin{cor}\label{cor:smooth-extension} 
Let $0\to V_1\to V_2\to V_3\to 0$ be a short exact sequence in $\Rep_k(G)$. If
$V_1$ and $V_3$ are smooth, then so is $V_2$.
\end{cor} 
\begin{proof} 
By Corollary~\ref{cor:smoothkernel} the induced sequence $0\to V_1\to V_2^\sm
\to V_3\to 0$ is exact, whence $V_2 = V_2^\sm$.
\end{proof} 

We have seen in Example~\ref{ex} that Corollary~\ref{cor:smooth-extension} may
fail if either $G$ does not admit a finitely generated open subgroup or $k$ has
characteristic $0$.

\begin{cor} 
One has $\Ext^1_{\Rep_k^\sm(G)}(V,W) = \Ext^1_{\Rep_k(G)}(V,W)$, for all $V,W\in
\Rep_k^\sm(G)$.
\end{cor} 
\begin{proof} 
This is a restatement of Corollary~\ref{cor:smooth-extension}.
\end{proof} 

\noindent\textbf{Acknowledgments.} 
This research was funded by the University of M\"unster and Germany's Excellence
Strategy EXC 2044 390685587, Mathematics M\"unster:
Dynamics--Geometry--Structure. I thank Emanuele Bodon for providing
Example\ref{ex}. My thanks extends to Peter Schneider for his interest in my
work. 

\bibliographystyle{alphaurl}
\bibliography{references}{}
\end{document}